\documentclass[11pt]{amsart}
\usepackage{graphicx}
\usepackage{mathpazo}
\usepackage{amssymb}
\usepackage{epstopdf}
\usepackage[cmtip,all]{xy}
\usepackage{hyperref}
\newtheorem{prop}{Proposition}
\newtheorem{lemma}[prop]{Lemma}
\newtheorem{thm}[prop]{Theorem}

\usepackage{amsmath}
\usepackage{mathrsfs}
\usepackage{amsthm}
\usepackage{amssymb}
\usepackage{amsfonts}
\usepackage{tikz} 
\usepackage{mathtools}
\usepackage{setspace}
\usepackage{cancel}
\usepackage{chemarrow}
\usetikzlibrary{matrix,arrows} 
\usepackage{amscd}

\newtheorem{theorem}{Theorem}

\newtheorem{assumption}{Assumption}
\theoremstyle{definition}
\newtheorem{example}{Example}

\newtheorem{defn}[prop]{Definition}

\newcommand{\CC}{\mathbb{C}}
\renewcommand{\P}{\mathbb{P}}

\newcommand{\D}{\mathcal{D}}

\newcommand{\I}{\mathcal{I}}

\renewcommand{\O}{\mathcal{O}}
\newcommand{\Z}{\mathbb{Z}}
\newcommand{\Q}{\mathbb{Q}}
\newcommand{\F}{\mathcal{F}}

\newcommand{\G}{\mathcal{G}}

\newcommand{\Ext}{\mathcal{E}xt}

\newcommand{\M}{\mathcal{M}}
\newcommand{\N}{\mathcal{N}}
\newcommand{\hilb}{\operatorname{Hilb}}

\newcommand{\Ex}{\operatorname{Ext}}
\newcommand{\Home}{\mathcal{H}om}
\newcommand{\Hom}{\operatorname{Hom}}
\newcommand{\pic}{\operatorname{Pic}}
\newcommand{\ch}{\operatorname{ch}}
\newcommand{\ob}{\operatorname{Ob}}
\renewcommand{\tan}{\operatorname{Tan}}
\newcommand{\ext}{\operatorname{Ext}}

\newcommand{\et}{{\text{\'et}}}
\newcommand{\DT}{\operatorname{DT}}

\title[DT invariants and linear systems]{Donaldson-Thomas invariants, linear systems and punctual Hilbert schemes}
\author{Amin Gholampour and Artan Sheshmani}
\date{\today}                                           
\begin{document}
\maketitle

\begin{abstract}
We study certain DT invariants arising from stable coherent sheaves in a nonsingular projective threefold supported on the members of a linear system of a fixed line bundle. When the canonical bundle of the threefold satisfies certain positivity conditions, we relate the DT invariants to Carlsson-Okounkov formulas for the ``twisted Euler's number" of the punctual Hilbert schemes of nonsingular surfaces, and conclude they have a modular property.
\end{abstract}
\section{Introduction}

S-duality predicts that certain generating functions of DT invariants of semistable 2-dimensional sheaves inside a Calabi-Yau threefolds are modular  (see for example \cite{a100, DM, GST}). In \cite{GS} we studied these DT invariants for K3 fibration over curves (which are not necessarily Calabi-Yau) and still got a modular answer.  In this paper, we study a certain type of these DT invariants in some other special cases of non-Calabi-Yau geometries and show that they have modular properties. To do this, we express them in terms of integrals over the Hilbert scheme of points on nonsingular surfaces.  Carlsson-Okounkov \cite{CO} found an explicit formula for the generating series of the integrals that arise this way extending G\"otsche's formula for the generating series of Euler's numbers of the Hilbert schemes. We give several examples of threefolds for which our required conditions are all satisfied.

\section{Statement of the result}
Let $(X,\O(1))$ be a  nonsingular polarized threefold with $$H^1(\O_X)=0=H^2(\O_X)$$  and $L$ be a fixed line bundle on $X$ generated by its global sections and satisfying the following further conditions:

\begin{assumption} \label{asum0} Assume that $H^0(L\otimes K_X)=0=H^1(L \otimes K_X)$ and $$ -K_X\cdot L^2>L^3, \qquad -K_X \cdot L\cdot \O(1)>0.$$ 
 \end{assumption} 
 
We think of  the condition $-K_X\cdot L^2>L^3$ as saying that $-K_X$ is sufficiently positive with respect to $L$.  The condition $-K_X \cdot L\cdot \O(1)>0$ is immediate for example if $-K_X$ and $L$ are ample. In the following lemma we summarize some of  implications of Assumption \ref{asum0}: 
 
 \begin{lemma} \label{vanish} Let $S\in |L|$ be a general member. Then $S$ is a nonsingular surface and  \begin{enumerate} \item $H^{1}(\O_S)=0=H^2(\O_S)$, \item $2\le \dim |L|=h^0(L|_S)$, \item  $H^1(L|_S)=0=H^{2}(L|_S)$. \end{enumerate}
 \end{lemma}
\begin{proof} Since $|L|$ is base-point free by Bertini's theorem $S$ is nonsingular. The natural short exact sequence $0\to L^*\to \O_X\to \O_S\to 0$ gives the exact sequence for $i=1, 2$ $$H^i(\O_X)\to H^i(\O_S)\to H^{i+1}(L^*)\cong H^{2-i}(L\otimes K_X)^*.$$ By our assumptions the first and the last terms vanish and hence so does the middle term. 

Again since $|L|$ is base-point free $\dim |L|\neq 0$ and if $\dim |L|=1$ any two distinct members of the linear system do not intersect and so $L^2=0$,  which contradicts the first inequality in  Assumption \ref{asum0}. Therefore, we must have $\dim |L| \ge 2$.
Since $H^1(\O_X)=0$, the natural short exact sequence $0\to \O_X\to L\to L|_S\to 0$ gives the exact sequence $$0\to H^0(\O_X)\to H^0(L)\to H^0(L|_S)\to0, $$ which proves that $\dim |L|=h^0(L)-1=h^0(L|_S)$.

Next, consider the linear system  $|L|_S|$ on $S$. It has to be base-point free because $|L|$ is.  Let $C\in |L|_S|$ be a general member, which must be smooth by Bertini's theorem. By Serre duality and adjunction formula $$H^1(L|_C)\cong H^0(L^*\otimes K_C)^*\cong H^0(K_S|_C)^*\cong H^0(L\otimes K_X |_C)=0,$$ where the last vanishing is because $\deg(L\otimes K_X |_C)=K_X\cdot L^2+L^3<0$.  Also, $H^2(L|_C)=0$ for dimension reason. So applying cohomology to the natural short exact sequence $0\to \O_S\to L|_S \to L|_C\to 0 $ and using the assumption  $H^{i\ge 1}(\O_S)=0$ we see that $H^{i\ge 1}(L|_S)=0$ as claimed.
\end{proof}

We consider the moduli space of coherent sheaves in $X$, which are supported on the members of $|L|$. For this fix a Chern character vector \begin{equation} \label{ch} \ch=\big(\ch_0=0,\; \ch_1=L,\; \ch_2=\gamma,\; \ch_3=\xi\big )\in \oplus_{i=0}^3 H^{2i}(X,\Q).\end{equation} We denote the moduli space of Gieseker semistable sheaves (with respect to $\O(1)$) with Chern character $\ch$ by $\M(X,\ch)$. It is a projective scheme. The Hilbert polynomial of coherent sheaves with Chern character $\ch$ is of degree 2, and the coefficients of degree 2 and degree 1 terms are respectively  given by 
$$a_2=L\cdot \O(1)^2/2,\qquad a_1=\gamma \cdot \O(1)-L\cdot K_X\cdot \O(1)/2.$$

We make the following assumption on $\ch_1, \ch_2$ to ensure semistability implies stability (for any choice of $\ch_3$):

\begin{assumption}\label{asum1} 
Assume that $L$ and $\gamma$ are such that for any decomposition $L=L_1+L_2$ in which $L_1$ and $L_2$ are the class of nonzero effective divisors, and any $m\in \mathbb Z$ 
$$\frac{2m+(L_1-K_X)\cdot L_1\cdot \O(1)}{L_1\cdot \O(1)^2}\neq \frac{a_1}{a_2}.$$

For example, when $L$ is an irreducible class the condition above is immediate for any $\gamma$.
\end{assumption}

\begin{lemma} \label{noss}
 $\M(X,\ch)$ contains no strictly semistable sheaves.
\end{lemma}
\begin{proof} Suppose $F\in \M(X,\ch)$ is strictly semistable. Then there exists a quotient $G$ of $F$ such that $\ch_1(G)=L_1$ and $\ch_2(G)=\gamma_1$, and $L-L_1$ is nonzero and effective such that $$\frac{\gamma_1\cdot \O(1)-L_1\cdot K_X\cdot \O(1)/2}{L_1\cdot \O(1)^2/2}=\frac{a_1}{a_2}.$$ Here, the left hand side is the ratio of coefficients of degree 1 and 2 terms of Hilbert polynomial of $G$.
But $\gamma_1=c_2(G)+L_1^2/2$ and $m:=c_2(G)\cdot \O(1) \in \mathbb Z$, so we get $$\frac{2m+(L_1-K_X)\cdot L_1\cdot \O(1)}{L_1\cdot \O(1)^2}=\frac{a_1}{a_2}$$ contradicting the condition in Assumption  \ref{asum1}.
\end{proof}

By Serre duality $$\Ex^3(F,F)\cong \Hom(F,F\otimes K_X)^*=0$$ for any closed point $F \in \M(X,\ch)$. This is because by the inequality $K_X \cdot L\cdot \O(1)<0$ the Hilbert polynomial of $F$ is greater than that of $F\otimes K_X$ and so by the stability of $F$ there is no nontrivial homomorphism from $F$ to $F\otimes K_X$. Therefore,

\begin{thm}(Thomas \cite[Theorem 3.30]{T}). Under Assumption \ref{asum1}, $\M(X,\ch)$ carries a perfect obstruction theory and hence a virtual class $$[\M(X,\ch)]^{vir}\in A_{v}(\M(X,\ch)), \qquad v:= -K_X\cdot L^2/2+1.$$

\end{thm} \qed

An interesting feature for us is that the virtual dimension $v$ only depends on $\ch_1=L$ (and not on $\ch_2, \ch_3$) and is equal to the dimension of the linear system $|L|$:


\begin{lemma} \label{virdim} $\dim |L|= v= -K_X\cdot L^2/2+1$.
\end{lemma}
\begin{proof}
The virtual dimension $v$ does not depend on $\ch_2$ and $\ch_3$, so by choosing them suitably  we may assume that $\O_S$ is a closed point of $\M(X,\ch)$ for some general member $S\in |L|$. Note that $S$ is in particular irreducible so $\O_S$ is therefore stable. Now we can write 
$$v=\text{ext}^1(\O_S,\O_S)-\text{ext}^2(\O_S,\O_S)=1-\chi(\O_S,\O_S).$$
Applying $\Hom(-,L|_S)$ to the natural short exact sequence $0\to \O_X\to L\to L|_S\to 0$  and taking Euler characteristics, we get 
$$\chi(\O_S,\O_S)=\chi(\O_S)-\chi(L|_S)=h^0(\O_S)-h^0(L|_S)=1-h^0(L|_S).$$ Here the second equality is because of the vanishing $$h^{i\ge 1}(\O_S)=0=h^{i\ge 1}(L|_S)$$ by Assumption \ref{asum0} and Lemma \ref{vanish}. So we showed that $v=h^0(L|_S)=\dim |L|$ by Lemma \ref{vanish}.
\end{proof}

Suppose now that $\ch$ is as in \eqref{ch}, and Assumptions \ref{asum0}, \ref{asum1} are satisfied. We define the Donaldson-Thomas invariant $\DT(X,\ch)$ as follows. Let $$\rho:\M(X,\ch)\to |L|, \qquad \rho(F):=\text{Div}(F),$$ where $\text{Div}(F)$ is the divisor associated to the coherent sheaf $F$ in the sense of  \cite{F, KM}. To see $\rho$ is a morphism of schemes one uses  modular properties of its source and target and the fact that the construction of $\text{Div}(-)$ is well-behaved under basechange, unlike the construction of scheme theoretic support of sheaves (see \cite{G} for details). Define $$\DT(X,\ch)=\rho_*\; [\M(X,\ch)]^{vir} \in A_v(|L|)\cong \Z.$$ 
We put these invariants into a generating function. For this, let $S$ be a general member of  $|L|$, and  \begin{equation}\beta \in H^2(S,\Z) \quad \text{such that} \quad  \label{beta}(i_* \beta^{PD})^{PD}=\gamma+L^2/2,\end{equation} where $PD$ denotes the Poincar\'{e} dual, and $i:S\hookrightarrow X$ is the inclusion. Given $\beta$ and $\ch$ as above, define  \begin{equation} \label{n} n=n(\beta, \ch)=\beta^2/2+\gamma \cdot L/2+2L^3/3-\xi.\end{equation}

\begin{lemma} \label{int} If $\M(X,\ch)\neq \emptyset$ then $n$ is a nonnegative integer for any choice of $\beta$ satisfying \eqref{beta}. For a given $n$ and $\ch$, there are finitely many class $\beta$ satisfying \eqref{beta} and \eqref{n}.
\end{lemma}
\begin{proof}
Let $F \in \M(X,\ch)$ be a closed point corresponding to a coherent sheaf supported on the general member $S\in |L|$ as above, such that $\ch(F)=\ch$. Then, $S$ is nonsingular and $F$ is rank 1 on its support, so  $F\cong i_*(N\otimes I)$, where $i\colon S\hookrightarrow X$ is the inclusion, $N$ is a line bundle, and $I$ is the ideal sheaf of colength $n$.  By GRR, $$\ch(i_*(N\otimes I))=i_*\left(\ch(N\otimes I)\cdot (1-L/2+L^2/6)|_S\right),$$ so $c_1(N)=\beta$ and $c_2(I)=n$  are respectively given by \eqref{beta} and \eqref{n}. This proves the first part.

Suppose $n$, $L$ and $\gamma$ are fixed. Since $H^{i\ge 1}(\O_S)=0$ we know that $\text{Pic}(S)\cong H^2(S,\Z)$ is a finitely generated abelian group. Clearly there are finitely many contributions from its torsion part. For the contributions of the free part,  by Hodge index theorem we can find a basis $e_1,\dots, e_s$ for $H^2(S,\mathbb R)$ such that $e_1^2=1$ and $e_i^2=-1$ for $i\ge 2$ and $e_i\cdot e_j=0$ for $i\neq j$ and $e_1$ is a (real) multiple of the class of $\O(1)|_S$.  Any $\beta \in H^2(X,\mathbb R)$ can be written as $\beta=\sum_{i=1}^s \alpha_i e_i$ where $(\alpha_1,\dots, \alpha_s)$ varies in a lattice in $\mathbb R^s$. Since $\gamma$ and $L$ are fixed, \eqref{beta} implies that $\alpha_1$ must be fixed. \eqref{n} then implies that $\sum_{i=2}^s\alpha_i^2$ is fixed, and hence there are finitely many choices for $\alpha_i$s. 
\end{proof}

\begin{defn} \label{gen}
Fix $L$ and $\gamma$ satisfying Assumption \ref{asum1}. Then define the generating series of DT invariants by summing over all possible $\xi$ (equivalently, over all nonnegative integers $n$ by Lemma \ref{int}): $$\DT_{L,\gamma}(X,q)=\sum_{\beta \text{ as in \eqref{beta}}}q^{\beta^2/2+\delta(L)/24} \; \sum_{n=0}^\infty \DT(X,\ch) \;q^{n},$$ where $\delta(L)=e(S)-K_S\cdot L+L^3$, and so only depends on the linear system $|L|$. By Lemma \ref{int} this generating series is well-defined (i.e. each power of $q$ corresponds to a sum of finitely many $\DT(X,\ch)$s).
\end{defn}

The main result of the paper is

\begin{theorem} \label{eta}
Given $(L,  \gamma)$ as in Definition \ref{gen}, $$\DT_{L,\gamma}(X,q)=\sum_{\beta \text{ as in \eqref{beta}}}q^{\beta^2/2} \;\eta(q)^{-\delta(L)},$$ where $\displaystyle \eta(q)=q^{1/24}\prod_{k>0}(1-q^k)$ is the eta function and $q=e^{2\pi i \tau}$. \end{theorem}

\section{Proof of Theorem \ref{eta}}

Suppose that $S \in |L|$ is a general member, and $n$ is related to $\ch$ by \eqref{beta} and \eqref{n}. Let $S^{[n]}$ be the Hilbert scheme of $n$ points on $S$. Repeating the argument of Lemma \ref{int} and using assumption $H^1(\O_S)=0$, we see that $\rho^{-1}(S)$ is isomorphic to a finite disjoint union of $m$ copies of $S^{[n]}$, one for each $\beta$ as in \eqref{beta} (Lemma \ref{int}).  $S^{[n]}$ is nonsingular of dimension $2n$. We show that $\rho\colon \M(X,\ch)\to |L|$ is smooth of relative dimension $2n$ over an open neighborhood $V\subset |L|$ of $S$ consisting of nonsingular divisors. 

Let $\D \subset X\times V$ be the (restriction of the) universal divisor, and  $$\hilb^n(\D/V), \qquad \pic(\D/V)$$ be respectively relative Hilbert scheme of length $n$ subschemes and the relative Picard scheme of $\D/V$. Since $\D/V$ have nonsingular fibers of dimension 2, $\hilb^n(\D/V)$ is nonsingular of dimension $2n+v$. By \cite[Theorem 4.8]{K}, $\pic(\D/V)$ represents the sheafification of the Picard functor in \'etale topology. Moreover, since by our assumption $h^{0,1}=h^{0,2}=0$ for the fibers of $\D/V$, by \cite[Theorem 5.19]{K}  $\pic(D/V)$ is smooth of relative dimension 0 and locally of finite type. Let $\pic_{\gamma}(\D/V)$ be the union of components, whose $S$-fiber is one of $\pic_\beta(S)$ where $\beta$ is as in \eqref{beta}.  By Lemma \ref{int}, $\pic_{\gamma}(\D/V)$ is of finite type. 

Let $\M_V:=\rho^{-1}(V)$. We claim that $$\M_V\cong \hilb^n(\D/V) \times_V \pic_\gamma(\D/V).$$ 
To see this, let $B$ be a scheme over $V$. If $\F$ is a flat family over $X\times B$ define $$\D_B:=\text{Div}(\F) \xhookrightarrow{i} X \times B.$$  Since $\text{Div}(-)$ is well-behaved with respect to basechange \cite{F, KM}, $\D_B\cong \D \times_V B$. We can then regard $\F$ as $i_*\G$ for some rank 1 torsion free sheaf $\G$ over $D_B/V$.  By construction $\G$ is flat and $\D_B$ is smooth over $B$. Taking double duals $\G\subset \G^{**}$, by \cite[Lemma 6.13]{Ko} $\N:=\G^{**}$ is a line bundle, and hence we get a family of ideal sheaves $\G(-\N)\subset \O_{\D_B}$ flat over $B$ with each $B$-fiber is an ideal of colength $n$. $\G(-\N)$ and $\N$ therefore determine  a $B$-valued point of $$\hilb^{n}(\D/V)\times_V \pic_\gamma(\D/V).$$ 
Conversely, if a pair $(\I,\N)$, of $B$-flat families of ideals of colength $n$ and line bundles of class $\gamma$ on the fibers of $\D_B:=\D\times_V B$ is given then $i_*(\I\boxtimes \N)$ determines a $B$-valued point of $\M_V$. These assignments are evidently inverse of each other, and so the claim is proven.   

Note that the Poincar\'e line bundles exist only after an \'etale basechange. Therefore, by construction above there is  an \'etale basechange 
$$
\xymatrix{\M_{V^\et} \ar[d]^-\rho \ar[r] &\M_V \ar[d]^-\rho\\ V^{\text{\'et}} \ar[r] &V}
$$ such that the universal sheaf exists over $X\times \M_{V^{\text{\'et}}}$. Let's denote it by $\F$ again and let $\D_{V^{\text{\'et}}}:=\text{Div}(\F) \xhookrightarrow{i} X \times V^\et$. Also, denote by $\I$ the universal ideal sheaf over $\D_{V^\et} \times_{V^\et}\hilb^{n}(\D_{V^\et}/V^\et)$.

%

Our goal is to write $\DT(X,\ch)$ as an integral over $S^{[n]}$. For this, we find a relation between tangent/obstruction theory of $\M_V$ and $S^{[n]}$ using the identification above. Note that $\pic_\gamma(\D/V)$ consists of finitely many nonsingular components each isomorphic to $V$. We can work over each connected component of $\M_V$ at a time, which is thus identified with $\hilb^{n}(\D/V)$. This shows in particular each connected component of $\M_V$ is nonsingular of dimension $2n+v$. In the rest of the proof we restrict our attention on one such fixed component.

We denote the tangent and the obstruction sheaves (of the fixed component) of $\M_V$ by $\tan$ and $\ob$, respectively. 
If $\pi \colon X\times \M_{V^{\text{\'et}}}\to \M_{V^{\text{\'et}}}$ is the projection then  
 $$ \Ext^1_\pi(\F,\F), \qquad  \Ext^2_\pi(\F,\F)$$ descend to $\tan$ and $\ob$, respectively.  These are the first and second cohomologies of the perfect complex $R\Home_{\pi}(\F,\F)$. We claim $\ob$ is  locally free. The nonsingular divisor $\D_{V^{\text{\'et}}}\xhookrightarrow{i} X\times V^{\text{\'et}}$ and adjunction gives an exact triangle \cite[Corollary 11.4]{H}
$$R\Home_{\pi'}(\I,\I)\to R\Home_{\pi}(\F,\F)\to R\Home_{\pi'}(\I,\I(\D_{V^\et}))[-1],$$ where $\pi'=i \circ \pi$. Taking cohomology, we get the exact sequence
$$0=\Ext^2_{\pi'}(\I,\I)\to \Ext^2_\pi(\F,\F)\to \Ext^1_{\pi'}(\I,\I(\D_{V^\et}))\to 0$$ in which the first vanishing is by basechange and our assumption $H^2(\O_S)=0$ for all closed $S\in V$ that in turn implies $\ext^2_S(I,I)=0$ for any colength $n$ ideal $I\subset \O_S$. 
Next, we show that  $\Ext^1_{\pi'}(\I,\I(\D_{V^\et}))$ is locally free of rank $2n$ and hence the same will be true for $\Ext^2_\pi(\F,\F)$ by the exact sequence above.
To do this we use basechange again and show that  over any closed point $S\in V$ the cohomologies of $R\Home_{\pi'}(\I,\I(\D_{V^\et}))$, given by  $\ext^{i}_S(I,I\otimes L)$ for colength $n$ ideals $I\subset \O_S$ have constant dimensions. For $i>2$ they are 0 because $S$ is nonsingular of dimension 2. For $i=2$, it is 0 because of $H^2(L|_S)=0$ (Lemma \ref{vanish}). Finally, $\Hom_S(I,I\otimes L)\cong H^0(L)$ has dimension $v+1$. Therefore, by Riemann-Roch $\ext^{1}_S(I,I\otimes L)$ has  dimension $2n$. This proves the claim.

 Now (any connected component of) $\M_V$ is nonsingular of dimension $2n+v$ and the obstruction sheaf $\ob$ is locally free of rank $2n$. By \cite[Proposition 5.6]{BF}, $$[\M_V]^{vir}=c_{2n}(\ob)\cap [\M_V].$$
The exact sequence above together with Lemma \ref{vanish} also show that the restriction of the obstruction bundle $\ob$ to  (any component of) the fiber $\rho^{-1}(S)$ is equivalent to the Carlson-Okounkov K-theory element $$Rp_*L -R\Home_p(\I,\I\otimes L)$$  where $p\colon S\times S^{[n]}\to S^{[n]}$ is the projection.  
 
Now let $j\colon \text{Spec}\; \CC \hookrightarrow |L|$ be the inclusion corresponding to the general member  $S\in |L|$, and form the fibered square
$$
\xymatrix{\rho^{-1}(S)  \ar[d]^-\rho \ar@{^(->}[r] & \M(X,\ch) \ar[d]^-\rho\\  \text{Spec}\; \CC \ar@{^(->}[r]^-j &|L|.}
$$
 By \cite[Theorems 6.2, 6.5]{Fu}, we can write 
\begin{align*} \DT(X,\ch)&=j^*\rho_* [\M(X,\ch)]^{vir}=\rho_*j^! [\M(X,\ch)]^{vir}\\
&=\rho_*j^! [\M_V]^{vir}=\rho_*j^! c_{2n}(\ob)\cap [\M_V]\\ 
&=\sum_{\beta \text{ as in \eqref{beta}}} \int_{S^{[n]}} c_{2n}(Rp_*L -R\Home_p(I,I\otimes L)).
\end{align*} 
These integrals are given by \cite[Corollary 1]{CO}. This completes the proof of Theorem.

\section{Examples}
\begin{example}[hypersurfaces in $\P^4$]
Let $X\subset \P^4$ be a nonsingular hypersurface of degree $d\le 3$ with the choices $$\O(1):=\O_{\P^4}(1)|_X=:L.$$ Then $H^{i\ge 1}(X,\O)=0$, and 
we have $K_X\cong \O(d-5)$, so $-K_X\cdot L\cdot \O(1)=5-d>0$ and $-K_X \cdot L^2= (5-d)d>d=L^3$ are satisfied for $d\le 3$. 

Moreover, by the Lefschetz hyperplane theorem $L$ is an irreducible class. Therefore, both Assumption \ref{asum0} and Assumption \ref{asum1} (for any $\gamma$) are satisfied for these three geometries. Note that if $d=4$, the first inequality in Assumption \ref{asum0} is not satisfied but the second one is. When $d\ge 5$ none of the inequalities are satisfied. For $d= 4,5$ the DT invariants of $X$ are still defined but they don't fit into the framework of this paper. When $d\ge 6$ even the DT invariants of $X$ are not defined. Here, by DT invariants we mean $\DT(X,\ch)$ for $\ch=(0,L, \gamma, \xi)$. 

In the rest of this example, we consider the case $d=2$. By Lefschetz hyperplane theorem  $$H^2(X,\mathbb Z)\cong \mathbb Z \cong \pic(X),\quad H^4(X,\mathbb Z)\cong \mathbb Z $$ are respectively generated  by the class of $L$ and the class of line $\ell=L^2/2$. Tensoring by $\otimes L^{\pm 1}$ induces an isomorphism $\M(X,\ch)\cong \M(X,\ch')$, where $\ch_1=L=\ch'_1$ and $\ch_2=\ch'_2\mp 2\ell$. Because of this identification, we have only two different generating series of DT invariants $$\DT_{L,\ell}(X,q), \quad \DT_{L,2\ell}(X,q).$$ 
A general member of $|L|$ is a nonsingular quadratic surface in $\P^3$ and so is isomorphic to $\P^1\times \P^1$. Let $e_1, e_2$ be the generators of $\pic(\P^1\times \P^1)$.  According to  \eqref{beta}  for any $k\in \mathbb Z$, the classes $\beta=k(e_1-e_2)$ correspond to $\gamma=\ell$ and the classes $\beta=e_1+k(e_1-e_2)$ correspond to $\gamma=2\ell$. By Theorem \ref{eta} 

\begin{align*}
 \DT_{L,\ell}(X,q)&=\sum_{k\in \Z} q^{-k^2}\; \eta(q)^{10},\\
\DT_{L,2\ell}(X,q)&=\sum_{k\in \Z}q^{-k^2-k}\; \eta(q)^{10}.
\end{align*}

\end{example}

\begin{example}[Blow ups of $\P^3$] Let $X$ be the blow up of $\P^3$ at finite number of disjoint points and lines. Let $E$ be the exceptional divisor and $L$ be the (pullback of) hyperplane class. Take $\O(1)=kL-E$ for some $k\gg 0$ for the polarization. $X$ is birational to $\P^3$ and hence $H^1(X,\O)=H^2(X,\O)=0$.

$K_X=-4L+E$, so $-K_X\cdot L^2=4>1=L^3$ and $-K_X\cdot L\cot \O(1)=4k-1>0$. Suppose $Q\subset X$ is the prober transform of a general cubic surface. Let $0\to \O(-3L+E)\to \O(E)\to \O_Q(E)\to 0$ be the natural short exact sequence. Then $H^0(\O(E)\cong H^0(\O_Q(E))\cong \CC$, and hence \begin{align*} H^0(L\otimes K_X)& \cong H^0( \O(-3L+E)=0, \\H^1(L\otimes K_X)&\cong H^1( \O(-3L+E)\cong H^1(\O(E))=0.\end{align*} To see that last vanishing, note that $E$ is a disjoint union finite copies of $\P^2$ and $\P^1\times \P^1$, and over each copy $\O_E(E)$ is either $\O_{P^2}(-1)$ or $\O_{\P^1\times \P^1}(-\Delta)$, so $H^1(\O_E(E))=0$. Together with the natural short exact sequence $0\to \O\to \O(E)\to \O_E(E)\to 0$ we conclude that $H^1(\O(E))=0$. We have checked that condition of  Assumption \ref{asum0} are satisfied for this geometry. 

For simplicity, we only verify Assumption \ref{asum0} in the case $X$ is the blow up of $\P^3$ at one point or one line. In the case of blow up of a point, $a_2=L\cdot(kL-E)^2/2=k^2/2$  and $$a_1=\gamma \cdot \O(1)+L\cdot (4L-E)\cdot (kL-E) /2=rk/2+s+2k,$$ where $\gamma=rL^2/2+sE^2$ for some $r, s\in \mathbb Z$. In the case of blow up of a line, $a_2=L\cdot(kL-E)^2/2=k^2/2-k/2$  and $$a_1=\gamma \cdot \O(1)+L\cdot (4L-E)\cdot (kL-E) /2=rk/2+s_1+s_2+2k-1/2,$$ where $\gamma=rL^2/2+s_1e_1+s_2 e_2$ for some $r, s_1,s_2\in \mathbb Z$, and $e_1, e_2$ are the generators of $H^2(E)$.

 Now if $L_1=E$ the condition is for any $m\in \mathbb Z$\begin{align*}&2m\neq \frac{rk+2s+4k}{k^2}\quad \text{(for blow up at a point)},\\
&m+2\neq \frac{(k+1)(rk+2s_1+2s_2+4k-1)}{k^2-k} \quad \text{(for blow up at a line)} .\end{align*} So for a given $\gamma$, for example it suffices to pick $k\gg 0$ such that the right hand side is not an integer.


%
   
\end{example}

\section*{Aknowledgement}
We thank Richard Thomas for providing valuable comments on the original draft of this paper. The second author was supported partially by the NSF DMS-1607871, NSF DMS-1306313, the Simons 38558, and Laboratory of Mirror
Symmetry NRU HSE, RF Government grant, ag. No 14.641.31.0001. The second author would like to further sincerely thank the Center for Mathematical Sciences and Applications at Harvard University, the center for Quantum Geometry of Moduli Spaces at Aarhus University, and the Laboratory of Mirror Symmetry in Higher School of Economics, Russian federation, for the great help and support.

\vspace{.5cm}

\noindent \textit{amingh@umd.edu\\ Department of  Mathematics, University of Maryland,  College Park, MD 20742-4015.}

\vspace{.5cm}
\noindent \textit{artan@cmsa.fas.harvard.edu\\
Center for Mathematical Sciences and Applications, Harvard University, Department of Mathematics, 20 Garden Street, Room 207, Cambridge, MA 02139.\\ 
Centre for Quantum Geometry of Moduli Spaces, Aarhus University, Department of Mathematics, NY Munkegade 118, Building 1530, 319, 8000, Aarhus C, Denmark.\\
National Research University Higher School of Economics, Russian Federation, Laboratory of Mirror Symmetry, NRU HSE, 6 Usacheva Street,
Moscow, Russia,  119048. }


\begin{thebibliography}{10}

\bibitem[BF]{BF} K.~Behrend, B.~Fantechi. The intrinsic normal cone. Invent. Math. 128:45--88 (1997).      
  
  \bibitem[CO]{CO} E.~Carlsson, A.~Okounkov, Exts and vertex operators. Duke Math. J. 161:1797--1815 (2012).         




\bibitem[DM]{DM} F.~Denef, G.~Moore. Split states, entropy enigmas, holes and halos. J.  High Energy Phys. 11 (129) (2011).

\bibitem[F]{F} J.~Fogarty. Truncated Hilbert Functors. Journal f\"ur die reine und angewandte Mathematik. 234:65--88 (1969).

\bibitem[Fu]{Fu} W.~Fulton, \textit{Intersection theory}, Springer-Verlag (1998).

\bibitem[G]{G} A.~Gholampour. Counting stable sheaves on singular curves and surfaces. In preparation.

\bibitem[GS]{GS} A.~Gholampour, A.~Sheshmani. Donaldson-Thomas invariants of 2-dimensional sheaves inside threefolds and modular forms. Advances in Mathematics. 2018; 326:79--107.

%
%

\bibitem[GST]{GST} A.~Gholampour, A.~Sheshmani, R.~P.~Thomas. Counting curves on surfaces in Calabi-Yau threefolds. Math. Ann. 360:67--78  (2014).

\bibitem[H]{H} D.~Huybrechts. Fourier-Mukai transforms in algebraic geometry. Oxford University Press on Demand; 2006 Apr 20.
    
\bibitem[K]{K} S.~L.~Kleiman. The picard scheme. arXiv preprint math/0504020 (2005).   
              
              
\bibitem[KM]{KM} F.~Knudsen, D.~Mumford. The projectivity  of the moduli space of stable curves I: preliminaries on Det and Div. Mathematica Scandinavica. 39(1):19--55 (1977).
              
              



  
  
\bibitem[Ko]{Ko} J.~Koll\'ar. Projectivity of complete moduli. Journal of Differential Geometry. 32(1):235--268 (1990).


\bibitem[OSV]{a100} H.~Ooguri, A.~Strominger, C.~Vafa. Black Hole Attractors and the
  Topological String. Physics Review D. 70 (10):106--119  (2001).
              
\bibitem[T]{T} R.~P.~Thomas. A holomorphic {C}asson invariant for {C}alabi-{Y}au 3-folds, and bundles on {$K3$} fibrations. J. Differential Geom, 54:367--438 (2000). 
              
              
\end{thebibliography}
\end{document}